\documentclass[12pt,letterpaper]{amsart}
\usepackage{amssymb,mathrsfs,bm}
\usepackage{hyperref}

\theoremstyle{plain}
\newtheorem{thm}{Theorem}
\newtheorem{lem}{Lemma}
\newtheorem*{claim}{Claim}
\newtheorem{conj}{Conjecture}
\theoremstyle{definition}
\newtheorem{example}{Example}

\allowdisplaybreaks[4]

\title[A cross-intersection problem with measures]
{A semidefinite programming approach to \\ a cross-intersection problem with measures}
\author{Sho Suda}
\address{Department of Mathematics Education, Aichi University of Education, Kariya 448-8542, Japan}
\email{suda@auecc.aichi-edu.ac.jp}
\author{Hajime Tanaka}
\address{Research Center for Pure and Applied Mathematics, Graduate School of Information Sciences, Tohoku University, Sendai 980-8579, Japan}
\email{htanaka@tohoku.ac.jp}
\author{Norihide Tokushige}
\address{College of Education, Ryukyu University, Nishihara, Okinawa 903-0213, Japan}
\email{hide@edu.u-ryukyu.ac.jp }

\begin{document}

\begin{abstract}
We present a semidefinite programming approach to bound the measures of cross-independent pairs in a bipartite graph. 
This can be viewed as a far-reaching extension of Hoffman's ratio bound on the independence number of a graph.
As an application, we solve a problem on the maximum measures of cross-intersecting families of subsets with two different product measures, which is a generalized measure version of the Erd\H{o}s--Ko--Rado theorem for cross-intersecting families with different uniformities.
\end{abstract}

\maketitle

\section{Introduction}

The Erd\H{o}s--Ko--Rado (EKR) theorem \cite{EKR1961QJMO} has many generalizations and extensions; see, e.g., \cite{DF1983SIAM,FG1989SDX,FT2003SSMH,Tokushige2005RMJ,GM2015B}.
In this paper, we focus on the following measure version due to Fishburn, Frankl, Freed, Lagarias, and Odlyzko~\cite{FFFLO1986SIAM} concerning intersecting families of subsets.
Let $n$ be a positive integer, and let $\Omega:=2^{[n]}$, where $[n]:=\{1,2,\dots, n\}$. 
We call $\bm{p}=(p^{(1)},\dots,p^{(n)})$ a \emph{probability vector} if $0<p^{(\ell)}<1$ for all $\ell\in[n]$.
Let $\mu_{\bm{p}}$ be the \emph{product measure} on $\Omega$ with respect to $\bm{p}$ defined by
\begin{equation*}
	\mu_{\bm{p}}(U):=\sum_{x\in U}\prod_{\ell\in x}p^{(\ell)}\!\prod_{k\in[n]\backslash x}\!(1-p^{(k)}) \qquad (U\subset\Omega).
\end{equation*}
Note that $\mu_{\bm{p}}$ is a probability measure on $\Omega$, i.e., $\mu_{\bm{p}}(\Omega)=1$.
We say that a family of subsets $U\subset \Omega$ is \emph{intersecting} if $x\cap y\neq\emptyset$ for all $x,y\in U$.

\begin{thm}[Fishburn et al.~\cite{FFFLO1986SIAM}]\label{FFFLO}
Let $\mu_{\bm{p}}$ be the product measure defined above, and assume that $p^{(1)}=\max\{p^{(\ell)}:\ell\in[n]\}$, and that $p^{(\ell)}\leqslant 1/2$ for $\ell\geqslant 2$.
If $U\subset\Omega$ is intersecting, then $\mu_{\bm{p}}(U) \leqslant p^{(1)}$.
Moreover, if $p^{(1)}>p^{(\ell)}$ for $\ell\geqslant 2$, then equality holds if and only if $U=\{x\in\Omega:1\in x\}$.
\end{thm}

We extend Theorem~\ref{FFFLO} to \emph{cross-intersecting} families.
Throughout the paper (except in Section \ref{sec: general problem}), we use the following notation.
Let $\Omega_1,\Omega_2$ be distinct copies of $\Omega$.
For $i=1,2$, let $\bm{p}_i=(p_i^{(1)},\dots,p_i^{(n)})$ be a probability vector.
For notational convenience, we set
\begin{equation*}
p_i:=p_i^{(1)}.
\end{equation*}
Let $\mu_i:=\mu_{\bm{p}_i}$ be the product measure on $\Omega_i$ with respect to $\bm{p}_i$.
We say that $U_1\subset \Omega_1,U_2\subset \Omega_2$ are \emph{cross-intersecting} if $x\cap y\neq\emptyset$ for all $x\in U_1,y\in U_2$.
For $i=1,2$, and $\ell\in[n]$, let
\begin{equation*}
	U_i^{(\ell)}:=\{x\in\Omega_i:\ell\in x\},
\end{equation*}
which is an intersecting family with $\mu_i(U_i^{(\ell)})=p_i^{(\ell)}$.

Our main result is as follows:

\begin{thm}\label{weighted-thm}
Let $\mu_1,\mu_2$ be the product measures defined above, and assume that
\begin{equation}\label{assumption on p}
	p_i=\max\{p_i^{(\ell)}:\ell\in [n]\} \qquad (i=1,2),
\end{equation}
and that $p_1^{(\ell)},p_2^{(\ell)}\leqslant 1/2$ for $\ell\geqslant 2$.
If $U_1\subset\Omega_1,U_2\subset\Omega_2$ are cross-intersecting, then
\begin{equation}\label{bound}
	\mu_1(U_1) \mu_2(U_2)\leqslant p_1p_2.
\end{equation}
Moreover, unless $p_1=p_2=1/2$ and $|w|\geqslant 3$, equality holds if and only if $U_1=U_1^{(\ell)},U_2=U_2^{(\ell)}$ for some $\ell\in w$, where
\begin{equation}\label{w}
	w=w_{\bm{p}_1,\bm{p}_2}:=\bigl\{\ell\in [n]:(p_1^{(\ell)},p_2^{(\ell)})=(p_1,p_2)\bigr\}.
\end{equation}
\end{thm}

EKR theorems for cross-intersecting families were previously studied, e.g., by Pyber \cite{Pyber1986JCTA} and Matsumoto and Tokushige \cite{MT1989JCTA} for uniform families of subsets, and by Suda and Tanaka \cite{ST2014BLMS} for uniform families of subspaces.
It should be remarked that if we apply Theorem \ref{weighted-thm} to the situation of Theorem \ref{FFFLO} with $\bm{p}_1=\bm{p}_2=\bm{p}$ and $U_1=U_2=U$, then the result is in fact stronger than Theorem \ref{FFFLO}, especially when $p^{(1)}<1/2$.
Theorem~\ref{weighted-thm} also generalizes several known results.
For example, Tokushige \cite{Tokushige2010JCT} obtained the bound \eqref{bound} under the following additional assumption:
\begin{equation*}
	\frac{1}{2} > p_1^{(1)}=\cdots=p_1^{(n)}, \quad \frac{1}{2} > p_2^{(1)}=\cdots=p_2^{(n)}.
\end{equation*}
Recently, Borg \cite{Borg2012pre} proved Theorem~\ref{weighted-thm} in the following two cases:
\begin{itemize}
\item[(i)] $p_1^{(\ell)}=p_2^{(\ell)}=1/m_{\ell}$ for some integers $m_{\ell}$ ($\ell\in[n]$) with $2\leqslant m_1\leqslant m_2\leqslant \cdots\leqslant m_n$;
\item[(ii)] $p_1^{(\ell)}=1/m_{\ell}$, $p_2^{(\ell)}=1/m'_{\ell}$ for some integers $m_{\ell},m'_{\ell}$ ($\ell\in[n]$) with $3\leqslant m_1\leqslant m_2\leqslant \cdots\leqslant m_n$ and $3\leqslant m'_1\leqslant m'_2\leqslant \cdots\leqslant m'_n$.
\end{itemize}
His proof is based on the shifting technique, and assuming that $\bm{p}_1,\bm{p}_2$ are decreasing sequences, as in (i) and (ii) above, seems inevitable to ensure that measures do not decrease while shifting.

We now comment on the exceptional cases in Theorem~\ref{weighted-thm}.
Indeed, when $p_1=p_2=1/2$ and $|w|\geqslant 3$, the pairs $U_1^{(\ell)},U_2^{(\ell)}$ with $\ell\in w$ are not necessarily the only cross-intersecting families having $p_1p_2=1/4$ as the product of measures.
The following two examples illustrate some other structures with this property:

\begin{example}\label{ex-3}
Assume that $n\geqslant 3$.
For $i=1,2$, let $U_i=\{x\in\Omega_i:|x\cap[3]|\geqslant 2\}$.
Then, $U_1,U_2$ are cross-intersecting. 
If $p_1^{(\ell)}=p_2^{(\ell)}=1/2$ for $\ell\in[3]$, then $\mu_1(U_1)\mu_2(U_2)=1/4$.
\end{example}

\begin{example}\label{ex-4}
Assume that $n\geqslant 4$.
For $i=1,2$, let $U_i=\{x\in\Omega_i:x\cap[4]\in C_i\}$, where
\begin{align*}
	C_1&:=\bigl\{\{1,2\},\{3,4\},\{1,3\},\{1,2,3\},\{1,2,4\},\{1,3,4\},\{2,3,4\},\{1,2,3,4\}\bigr\}, \\
	C_2&:=\bigl\{ \{1,4\},\{2,3\}, \{1,3\},\{1,2,3\},\{1,2,4\},\{1,3,4\},\{2,3,4\},\{1,2,3,4\}\bigr\}.
\end{align*}
Then, $U_1,U_2$ are cross-intersecting, but neither $U_1$ nor $U_2$ is intersecting. 
If $p_1^{(\ell)}=p_2^{(\ell)}=1/2$ for $\ell\in[4]$, then $\mu_1(U_1)\mu_2(U_2)=1/4$.
\end{example}

There might be several directions to extend Theorem~\ref{weighted-thm}.
The following conjecture would be one of the most interesting possible extensions in the sense that, if true, each family in the optimal case is intersecting, but not necessarily measure maximal.

\begin{conj}\label{relaxed version}
Theorem~\ref{weighted-thm} still holds if we replace the condition \eqref{assumption on p} with
\begin{equation}\label{weak assumption}
	p_1p_2=\max\bigl\{p_1^{(\ell)}p_2^{(\ell)}:\ell\in [n]\bigr\},
\end{equation}
where we understand in this case that $w=\bigl\{\ell\in [n]:p_1^{(\ell)}p_2^{(\ell)}=p_1p_2\bigr\}$.
\end{conj}

\noindent
The following result supports Conjecture~\ref{relaxed version}:

\begin{thm}\label{thm 1/3}
Let $\mu_1,\mu_2$ be the product measures defined above, and assume that \eqref{weak assumption} holds (instead of \eqref{assumption on p}), and that $p_1^{(\ell)},p_2^{(\ell)}\leqslant 1/3$ for $\ell\in[n]$.
If $U_1\subset\Omega_1,U_2\subset\Omega_2$ are cross-intersecting, then \eqref{bound} holds.
Moreover, equality holds in \eqref{bound} if and only if $U_1=U_1^{(\ell)}$, $U_2=U_2^{(\ell)}$ for some $\ell\in w$, where $w$ is as in Conjecture~\ref{relaxed version}.
\end{thm}

We also conjecture that the optimal structure in Theorem~\ref{weighted-thm} has a stability, namely, if $U_1\subset\Omega_1,U_2\subset\Omega_2$ are cross-intersecting and $\mu_1(U_1)\mu_2(U_2)$ is close to $p_1p_2$, then they are `close' to the pair $U_1^{(\ell)},U_2^{(\ell)}$ for some $\ell\in[n]$ in the following sense:

\begin{conj}\label{stability}
Let $\mu_1,\mu_2$ be the product measures as above, and assume that \eqref{assumption on p} holds.
 If $p_1,p_2<1/2$, then there exists a constant $c=c(p_1,p_2)$ such that, for any cross-intersecting families $U_1\subset\Omega_1,U_2\subset\Omega_2$ with 
\begin{equation*}
	\mu_1(U_1)\mu_2(U_2)>(1-\varepsilon)\,p_1p_2,
\end{equation*}
there is an $\ell\in[n]$ such that
\begin{equation*}
	\max\bigl\{\mu_1(U_1\triangle U_1^{(\ell)}), \mu_2(U_2\triangle U_2^{(\ell)}) \bigr\}<c\sqrt{\varepsilon}. 
\end{equation*}
\end{conj}

\noindent
See \cite{Friedgut2008C} for more about stability.
Conjecture~\ref{stability} is true if all the $p_i^{(\ell)}$ are the same; see \cite{Tokushige2013CPC}. 

Our proofs of Theorem \ref{weighted-thm} and Theorem \ref{thm 1/3} are algebraic in nature, and are based on semidefinite programming (SDP).
Algebraic/spectral techniques have been quite successful in proving EKR type results; see \cite{GM2015B}.
Let $G$ be a regular graph.
(All the graphs we consider in this paper are finite and simple.)
Let $\lambda_{\max}$ and $\lambda_{\min}$ be the largest and the smallest eigenvalues of the adjacency matrix of $G$.
Recall that $U\subset V(G)$ is \emph{independent} if no two vertices in $U$ are adjacent.
Hoffman's famous \emph{ratio bound} states that
\begin{equation*}
	\frac{|U|}{|V(G)|}\leqslant\frac{-\lambda_{\min}}{\lambda_{\max}-\lambda_{\min}}
\end{equation*}
provided $U$ is independent.
The ratio bound is derived easily by giving a feasible solution to (the dual of) an SDP problem defining Lov\'{a}sz's \emph{theta function} \cite{Lovasz1979IEEE}; see also \cite{Delsarte1973PRRS,Schrijver1979IEEE,GM2012B}.
Wilson \cite{Wilson1984C} used the theta function to refine the original EKR theorem \cite{EKR1961QJMO}.

Now, let $G$ be a biregular bipartite graph with bipartition $V(G)=\Omega_1\sqcup\Omega_2$ (where $\Omega_1,\Omega_2$ are arbitrary finite sets), so the degrees are assumed to be constant on each of $\Omega_1,\Omega_2$.
We say that $U_1\subset\Omega_1,U_2\subset\Omega_2$ are \emph{cross-independent} if there are no edges between $U_1$ and $U_2$.
Using the largest two singular values $\sigma_1\geqslant\sigma_2$ of the bipartite adjacency matrix (i.e., one of the off-diagonal blocks of the adjacency matrix) of $G$, we have
\begin{equation*}
	\sqrt{\frac{|U_1||U_2|}{|\Omega_1||\Omega_2|}}\leqslant\frac{\sigma_2}{\sigma_1+\sigma_2}
\end{equation*}
provided $U_1,U_2$ are cross-independent.
There are several cases where this upper bound is sharp; see \cite{Tokushige2013JAC,Tokushige2013CPC}.
This bound can again be strengthened via SDP, and Suda and Tanaka \cite{ST2014BLMS} used this approach to obtain a sharp EKR type bound for cross-intersecting families of subspaces mentioned earlier.

In this paper, the graph we shall consider is the \emph{bipartite disjointness graph} $G$ with bipartition $V(G)=\Omega_1\sqcup\Omega_2$, where $\Omega_1=\Omega_2=2^{[n]}$ as in Theorem \ref{weighted-thm}, and $x\in\Omega_1,y\in\Omega_2$ are adjacent if and only if $x\cap y=\emptyset$.
Thus, that $U_1\subset\Omega_1,U_2\subset\Omega_2$ are cross-intersecting is equivalent to saying that they are cross-independent in $G$.
In Section~\ref{sec: general problem}, we present a general scheme to bound the measures of cross-independent pairs in a bipartite graph using SDP; see Theorem~\ref{sdp-thm}.
In Section~\ref{sec: proof}, we apply the scheme to the bipartite disjointness graph, and prove Theorem~\ref{weighted-thm} by constructing an appropriate feasible solution to the dual of the corresponding SDP problem.
To be more precise, it follows that this approach is applicable to the case where $p_1,p_2\leqslant 1/2$, and when $p_1>1/2$ or $p_2>1/2$, we employ some combinatorial arguments to reduce to the former case.
In Section~\ref{sec: proof of thm3}, we include a proof of Theorem~\ref{thm 1/3}, which is a slight modification of the proof of Lemma~\ref{lem:ineq} from Subsection~\ref{sec: proof of bound when p_1 at most 1/2}.

We end Introduction with some comments about the assumption in Theorem \ref{FFFLO} that $p^{(\ell)}\leqslant 1/2$ for $\ell\geqslant 2$.
Indeed, we were unable to verify whether or not the condition $\ell\geqslant 2$ is best possible.
Here, we bravely conjecture the following:

\begin{conj}
Let $\mu_{\bm{p}}$ be the product measure as in Theorem~\ref{FFFLO}, and assume that $p^{(1)}=\max\{p^{(\ell)}:\ell\in[n]\}$, and that $p^{(\ell)}\leqslant 1/2$ for $\ell\geqslant 3$.
If $U\subset\Omega$ is intersecting, then $\mu_{\bm{p}}(U) \leqslant p^{(1)}$.
Moreover, if $p^{(1)}>p^{(\ell)}$ for $\ell\geqslant 3$, or if $p^{(1)}<1/2$, then equality holds if and only if $U=\{x\in\Omega:\ell\in x\}$ for some $\ell\in[n]$ with $p^{(\ell)}=p^{(1)}$.
\end{conj}

\noindent
If this conjecture is true, then the assumption that $p^{(\ell)}\leqslant 1/2$ for $\ell\geqslant 3$ is best possible, as the following example shows:
\begin{example}
Let $\bm{p}=(p^{(1)},\dots,p^{(n)})$ be such that $p:=p^{(1)}=p^{(2)}=p^{(3)}>1/2$, and let $U=\{x\in\Omega:|x\cap[3]|\geqslant 2\}$. 
Then $U$ is intersecting, and $\mu_{\bm p}(U)=p^3+3p^2(1-p)>p$. 
\end{example}

\section{An SDP problem for a general cross-intersection problem}
\label{sec: general problem}

We restate our problem of bounding the measures of cross-independent pairs in a bipartite graph as an SDP problem.
See, e.g., \cite{Todd2001AN,GM2012B} for more about SDP in general.

We list some notation we will use.
Let $\Omega_1,\Omega_2$ be non-empty finite sets, and let $\widehat{\Omega}:=\Omega_1\sqcup\Omega_2$.
Let $G$ be a (not necessarily biregular) bipartite graph with bipartition $V(G)=\widehat{\Omega}=\Omega_1\sqcup\Omega_2$.
For two vertices $x,y\in \widehat{\Omega}$, we write $x\sim y$ if and only if $x,y$ are adjacent.
For $i=1,2$, let $\mu_i$ be a probability measure on $\Omega_i$.

Let $\mathbb{R}^{\widehat{\Omega}\times \widehat{\Omega}}$ be the set of real matrices with rows and columns indexed by $\widehat{\Omega}$, and let $\mathbb{R}^{\widehat{\Omega}}$ be the set of real column vectors with coordinates indexed by $\widehat{\Omega}$.
The sets $\mathbb{R}^{\Omega_i\times\Omega_j}$ and $\mathbb{R}^{\Omega_i}$ are similarly defined.
Let $J_{i,j}\in\mathbb{R}^{\Omega_i\times\Omega_j}$ be the all ones matrix.
For $x\in\Omega_i,y\in\Omega_j$, let $E_{x,y}\in\mathbb{R}^{\Omega_i\times\Omega_j}$ be the matrix with a $1$ in the $(x,y)$-entry and $0$ elsewhere.
Let $\Delta_i\in \mathbb{R}^{\Omega_i\times\Omega_i}$ be the diagonal matrix whose $(x,x)$-entry is $\mu_i(\{x\})$ for $x\in \Omega_i$.
Let $S\mathbb{R}^{\widehat{\Omega}\times \widehat{\Omega}}$ be the set of symmetric matrices in $\mathbb{R}^{\widehat{\Omega}\times \widehat{\Omega}}$.
We write $Y \bullet Z := \mathrm{trace} (Y^{\mathsf{T}} Z)$ for $Y,Z\in \mathbb{R}^{\widehat{\Omega}\times \widehat{\Omega}}$.

Now, suppose that $U_1\subset\Omega_1,U_2\subset\Omega_2$ are cross-independent in $G$.
We are looking for an upper bound on $\mu_1(U_1)\mu_2(U_2)$, so that we may assume that $\mu_1(U_1)>0$ and $\mu_2(U_2)>0$ without loss of generality.
Let $\bm{x}_i\in\mathbb{R}^{\Omega_i}$ be the characteristic vector of $U_i$, and let
\begin{equation*}
	X=X_{U_1,U_2}:=\begin{bmatrix} \frac{1}{\sqrt{\mu_1(U_1)}}\bm{x}_1 \\ \frac{1}{\sqrt{\mu_2(U_2)}}\bm{x}_2 \end{bmatrix} \!\! \begin{bmatrix} \frac{1}{\sqrt{\mu_1(U_1)}}\bm{x}_1 \\ \frac{1}{\sqrt{\mu_2(U_2)}}\bm{x}_2 \end{bmatrix}^{\!\mathsf{T}}\in S\mathbb{R}^{\widehat{\Omega}\times\widehat{\Omega}}.
\end{equation*}
Note that $X$ is a feasible solution to the following SDP problem with objective value $\sqrt{\mu_1(U_1)\mu_2(U_2)}$:
\begin{equation*}
\begin{array}{lll}
	\text{(P):} & \text{maximize} & \begin{bmatrix} 0 & \frac{1}{2}\Delta_1J_{1,2}\Delta_2 \\ \frac{1}{2}\Delta_2J_{2,1}\Delta_1 & 0\end{bmatrix} \bullet X \\[0.12in]
		& \text{subject to} & \begin{bmatrix} \Delta_1 & 0 \\ 0 & 0 \end{bmatrix} \bullet X = \begin{bmatrix} 0 & 0 \\ 0 & \Delta_2 \end{bmatrix} \bullet X = 1, \\[.12in]
		&& \begin{bmatrix} 0 & E_{x,y} \\ E_{y,x} & 0 \end{bmatrix} \bullet X = 0 \ \ \text{for} \ \ x\in\Omega_1,\, y\in\Omega_2, \, x\sim y, \\
		&& X \succcurlyeq 0, \ X \geqslant 0,
\end{array}
\end{equation*}
where $X \in S\mathbb{R}^{\widehat{\Omega} \times \widehat{\Omega}}$ is the variable, and $X \succcurlyeq 0$ (resp.~$X \geqslant 0$) means that $X$ is positive semidefinite (resp.~nonnegative).
The dual problem is then given by
\begin{equation*}
\begin{array}{lll}
	\text{(D):} & \text{minimize} & \alpha + \beta \\
		& \text{subject to} & S:=\begin{bmatrix} \alpha \Delta_1 & -\frac{1}{2} \Delta_1 J_{1,2} \Delta_2 \\  -\frac{1}{2} \Delta_2 J_{2,1} \Delta_1 & \beta \Delta_2 \end{bmatrix} +{\displaystyle\sum_{x\sim y}}\,\gamma_{x,y} \begin{bmatrix} 0 & E_{x,y} \\ E_{y,x} & 0 \end{bmatrix} - Z \succcurlyeq 0, \\
		&& Z \geqslant 0,
\end{array}
\end{equation*}
where $\alpha,\beta,\gamma_{x,y}\in\mathbb{R}$, and $Z\in S\mathbb{R}^{\widehat{\Omega}\times \widehat{\Omega}}$ are the variables, and the sum is over $x\in\Omega_1$, $y\in\Omega_2$ with $x\sim y$.
Indeed, for any feasible solutions to (P) and (D), we have
\begin{align*}
	\alpha+\beta - \begin{bmatrix} 0 & \frac{1}{2}\Delta_1J_{1,2}\Delta_2 \\ \frac{1}{2}\Delta_2J_{2,1}\Delta_1 & 0\end{bmatrix} \bullet X &= \begin{bmatrix} \alpha \Delta_1 & -\frac{1}{2}\Delta_1J_{1,2}\Delta_2 \\ -\frac{1}{2}\Delta_2J_{2,1}\Delta_1 & \beta \Delta_2 \end{bmatrix} \bullet X \\
	&\geqslant \left( {\displaystyle-\sum_{x\sim y}}\,\gamma_{x,y} \begin{bmatrix} 0 & E_{x,y} \\ E_{y,x} & 0 \end{bmatrix} + Z \right)\! \bullet X \\
	&= Z \bullet X \\
	&\geqslant 0.
\end{align*}
In particular, $\alpha+\beta$ is an upper bound on $\sqrt{\mu_1(U_1)\mu_2(U_2)}$.
We also note that feasible solutions to (P) and (D) are \emph{both} optimal if and only if $S\bullet X=Z\bullet X=0$.
(By \cite[Theorem 4.1]{Todd2001AN}, there is no duality gap in this case.)
In summary, we have the following:

\begin{thm}\label{sdp-thm}
Let $G$ be a bipartite graph with bipartition $V(G)=\widehat{\Omega}=\Omega_1\sqcup\Omega_2$, and let $\mu_i$ be a probability measure on $\Omega_i$ for $i=1,2$. 
Suppose that $U_1\subset\Omega_1,U_2\subset\Omega_2$ are cross-independent in $G$.
If $(\alpha,\beta,\gamma_{x,y},Z)$ is feasible in (D), then
\begin{equation*}
	\mu_1(U_1) \mu_2(U_2)\leqslant (\alpha+\beta)^2.
\end{equation*}
\end{thm}

\section{Proof of Theorem~\ref{weighted-thm}}
\label{sec: proof}

Let $q_i^{(\ell)}:=1-p_i^{(\ell)}$ for $i=1,2$, and $\ell\in[n]$.
We also set $q_i:=q_i^{(1)}=1-p_i$.
Throughout the proof, we shall always assume without loss of generality that
\begin{equation*}
	p_1\geqslant p_2.
\end{equation*}

The proof of Theorem~\ref{weighted-thm} consists of four parts.
First, we prove the bound \eqref{bound} when $p_1\leqslant 1/2$, based on the SDP method developed in Section~\ref{sec: general problem}.
Second, we exploit the duality of SDP further to prove a lemma which enables us to reduce the proof of Theorem~\ref{weighted-thm} when $p_1\leqslant 1/2$ to the case where $w=[n]$.
Third, we complete the proof of  the theorem when $p_1\leqslant 1/2$.
Finally, we prove Theorem~\ref{weighted-thm} when $p_1>1/2$.

\subsection{Proof of the bound \eqref{bound} when $p_1\leqslant 1/2$}
\label{sec: proof of bound when p_1 at most 1/2}

Recall that the bipartite disjointness graph $G$ has bipartition $V(G)=\Omega_1\sqcup\Omega_2$ with $\Omega_i=2^{[n]}$ ($i=1,2$), and $x\in\Omega_1$, $y\in\Omega_2$ are adjacent if and only if $x\cap y=\emptyset$.
We apply Theorem~\ref{sdp-thm} to this graph with $\alpha=\beta=\sqrt{p_1p_2}/2$ to prove the following:

\begin{lem}\label{lem:ineq}
If $p_1\leqslant 1/2$, then \eqref{bound} holds.
\end{lem}

The rest of this subsection is devoted to the proof of Lemma~\ref{lem:ineq}, so we assume from now on that $p_1\leqslant 1/2$.
We will find appropriate $\gamma_{x,y}$ and $Z$ in two steps: the case $n=1$ and the general case $n\geqslant 2$.
It is somewhat surprising that the initial case is more essential, and the general case follows easily from a tensor product construction.

We start with the case $n=1$. 
In this case, we can quickly verify that the following is a feasible solution to (D) with objective value $\sqrt{p_1p_2}$:
\begin{equation*}
	\alpha=\beta=\frac{\sqrt{p_1p_2}}{2}, \quad \gamma_{\emptyset,\emptyset}=\frac{q_1q_2}{2}, \ \gamma_{\emptyset,\{1\}}=\frac{q_1p_2}{2}, \ \gamma_{\{1\},\emptyset}=\frac{p_1q_2}{2}, \quad Z=0.
\end{equation*}
This solution is obtained by simply letting the positive semidefinite matrix $S$ in (D) have as less nonzero entries as possible.
However, in order to construct feasible solutions for general $n$, we incorporate an idea of Friedgut \cite{Friedgut2008C}.

Let $c_i:=\sqrt{{p_i}/{q_i}}$, and define
\begin{equation*}
	A_{i,j}:= \begin{bmatrix} 1-\frac{p_j}{q_i} & \frac{p_j}{q_i} \\ 1 & 0 \end{bmatrix}, \ D_{i,j}:=\begin{bmatrix} 1 & 0 \\ 0 & -c_ic_j \end{bmatrix}, \ V_i:=\begin{bmatrix} 1 & c_i \\ 1 & -\frac{1}{c_i} \end{bmatrix}, \ \Delta_i:=\begin{bmatrix} q_i & 0 \\ 0 & p_i \end{bmatrix},
\end{equation*}
where the rows and columns are indexed in the order $\emptyset$, $\{1\}$.
Then, it follows that
\begin{equation}\label{AV=VD}
	A_{i,j}V_j=V_iD_{i,j}, \quad V_i^{\mathsf{T}}\Delta_iV_i=I \ \ \text{(the identity matrix)}, 
\quad (\Delta_iA_{i,j})^{\mathsf{T}}=\Delta_jA_{j,i},
\end{equation}
from which it follows that
\begin{equation}\label{VDelAV=D}
	V_i^{\mathsf{T}} (\Delta_iA_{i,j}) V_j = D_{i,j}.
\end{equation}
Note also that
\begin{equation}\label{VDelJDelV}
	V_i^{\mathsf{T}} (\Delta_i J_{i,j} \Delta_j) V_j = E_{\emptyset,\emptyset}.
\end{equation}

Now, we consider a symmetric matrix of the form
\begin{equation}\label{S}
	S=\begin{bmatrix} \frac{\sqrt{p_1p_2}}{2} \Delta_1-\varepsilon_1\Delta_1A_{1,1} & \eta\Delta_1A_{1,2} -\frac{1}{2} \Delta_1 J_{1,2} \Delta_2 \\ \eta\Delta_2A_{2,1} -\frac{1}{2} \Delta_2 J_{2,1} \Delta_1 & \frac{\sqrt{p_1p_2}}{2} \Delta_2 -\varepsilon_2\Delta_2A_{2,2}\end{bmatrix},
\end{equation}
where $\varepsilon_1,\varepsilon_2,\eta\in\mathbb{R}$ and $\varepsilon_1,\varepsilon_2\geqslant 0$.
In other words, we take
\begin{equation}\label{values}
	\alpha=\beta=\frac{\sqrt{p_1p_2}}{2}, \quad \sum_{x\sim y}\gamma_{x,y} E_{x,y}=\eta\Delta_1A_{1,2}, \quad Z=\begin{bmatrix} \varepsilon_1\Delta_1A_{1,1} & 0 \\ 0 & \varepsilon_2\Delta_2A_{2,2} \end{bmatrix}.
\end{equation}
Since $p_1\leqslant 1/2$ and $\varepsilon_1,\varepsilon_2\geqslant 0$, we have $Z\geqslant 0$.
Moreover, by \eqref{AV=VD}, \eqref{VDelAV=D}, and \eqref{VDelJDelV},
\begin{equation*}
	\begin{bmatrix} V_1^{\mathsf{T}} & 0 \\ 0 & V_2^{\mathsf{T}} \end{bmatrix}S \begin{bmatrix} V_1 & 0 \\ 0 & V_2\end{bmatrix}=\begin{bmatrix} \frac{\sqrt{p_1p_2}}{2}-\varepsilon_1 & 0 & \eta-\frac{1}{2} & 0 \\ 0 & \frac{\sqrt{p_1p_2}}{2}+\varepsilon_1 c_1^2 & 0 & -\eta c_1c_2 \\ \eta-\frac{1}{2} & 0 & \frac{\sqrt{p_1p_2}}{2}-\varepsilon_2 & 0 \\ 0 & -\eta c_1c_2 & 0 & \frac{\sqrt{p_1p_2}}{2}+\varepsilon_2c_2^2 \end{bmatrix}.
\end{equation*}
Thus, $S\succcurlyeq 0$ if and only if
\begin{equation*}
	\begin{bmatrix} \frac{\sqrt{p_1p_2}}{2}-\varepsilon_1 & \eta-\frac{1}{2} \\ \eta-\frac{1}{2} & \frac{\sqrt{p_1p_2}}{2}-\varepsilon_2 \end{bmatrix} \succcurlyeq 0, \quad \begin{bmatrix} \frac{\sqrt{p_1p_2}}{2}+\varepsilon_1 c_1^2 & -\eta c_1c_2 \\ -\eta c_1c_2 & \frac{\sqrt{p_1p_2}}{2}+\varepsilon_2 c_2^2 \end{bmatrix} \succcurlyeq 0,
\end{equation*}
if and only if
\begin{gather}
	0\leqslant\varepsilon_1,\varepsilon_2\leqslant \frac{\sqrt{p_1p_2}}{2}, \label{psd0} \\
	\left(\frac{\sqrt{p_1p_2}}{2}-\varepsilon_1 \right) \!\! \left(\frac{\sqrt{p_1p_2}}{2}-\varepsilon_2 \right) \geqslant \left(\eta-\frac{1}{2} \right)^2, \label{psd1} \\
	\left(\frac{\sqrt{p_1p_2}}{2}+\varepsilon_1 c_1^2 \right) \!\! \left( \frac{\sqrt{p_1p_2}}{2}+\varepsilon_2 c_2^2 \right) \geqslant \eta^2 c_1^2c_2^2. \label{psd2}
\end{gather}
Observe that the following form a one-parameter family of solutions to \eqref{psd0}, \eqref{psd1}, and \eqref{psd2}, which therefore provide feasible solutions to (D) with objective value $\sqrt{p_1p_2}$\,:
\begin{equation}\label{epsilon-eta}
	\varepsilon_1=\frac{p_2}{p_1}\varepsilon_2+\frac{(p_1-p_2)p_2}{2\sqrt{p_1p_2}}, \quad \eta=\frac{p_2}{\sqrt{p_1p_2}}\varepsilon_2+\frac{q_2}{2}, \quad 0\leqslant\varepsilon_2\leqslant\frac{\sqrt{p_1p_2}}{2}.
\end{equation}
(Recall that we are assuming that $p_1\geqslant p_2$.)
This proves the bound \eqref{bound} for the (trivial) case $n=1$.
We note that, for the solutions \eqref{epsilon-eta}, equality is attained in each of \eqref{psd1} and \eqref{psd2}.
We also note that just one solution from \eqref{epsilon-eta} is enough to prove \eqref{bound}.
However, we will make full use of the one-parameter family of solutions to determine the extremal structure.

Next, we consider the general case $n\geqslant 2$.
Let $c_i^{(\ell)}:=\sqrt{{p_i^{(\ell)}}/{q_i^{(\ell)}}}$, and let
\begin{equation*}
	A_{i,j}^{(\ell)}:=\begin{bmatrix} 1-\frac{p_j^{(\ell)}}{q_i^{(\ell)}} & \frac{p_j^{(\ell)}}{q_i^{(\ell)}} \\ 1 & 0 \end{bmatrix},
\end{equation*}
where the rows and columns are indexed in the order $\emptyset$, $\{\ell\}$.
We define
\begin{equation*}
	A_{i,j}=A_{i,j}^{[n]}:=A_{i,j}^{(n)}\otimes A_{i,j}^{(n-1)}\otimes\cdots\otimes A_{i,j}^{(1)}.
\end{equation*}
We naturally identify $2^{\{1\}}\times\dots\times 2^{\{n\}}$ with $2^{[n]}$, and thus we view the rows (resp.~columns) of $A_{i,j}$ as indexed by $\Omega_i$ (resp.~$\Omega_j$).
We also define $D_{i,j}=D_{i,j}^{[n]}$, $V_i=V_i^{[n]}$, and $\Delta_i=\Delta_i^{[n]}$ in the same manner.
Note in particular that $\Delta_i$ is the diagonal matrix whose $(x,x)$-entry is $\mu_i(\{x\})$ for $x\in\Omega_i$.
With these new matrices we have \eqref{AV=VD}, \eqref{VDelAV=D}, and \eqref{VDelJDelV}.
Observe also that the $(x,y)$-entry of $A_{i,j}$ is $0$ if $x\cap y\neq\emptyset$ ($x\in\Omega_i,y\in\Omega_j$).
Moreover, if $p_1<1/2$ then the $(x,y)$-entry is positive whenever $x\cap y=\emptyset$.

Again, we consider the symmetric matrix $S$ defined by \eqref{S} where $\varepsilon_1,\varepsilon_2$, and $\eta$ are given by \eqref{epsilon-eta}.
Namely, we choose the variables by \eqref{values} with the new matrices.
Then,
\begin{align*}
	\begin{bmatrix} V_1^\mathsf{T} & 0 \\ 0 & V_2^\mathsf{T} \end{bmatrix} S \begin{bmatrix} V_1 & 0 \\ 0 & V_2 \end{bmatrix} &= \begin{bmatrix} \frac{\sqrt{p_1p_2}}{2} I_1 -\varepsilon_1 D_{1,1} & \eta D_{1,2}-\frac{1}{2} E_{\emptyset,\emptyset} \\ \eta D_{2,1}-\frac{1}{2} E_{\emptyset,\emptyset} & \frac{\sqrt{p_1p_2}}{2} I_2 -\varepsilon_2 D_{2,2} \end{bmatrix} \cong \bigoplus_{z\in 2^{[n]}} S^{(z)},
\end{align*}
where $I_i\in\mathbb{R}^{\Omega_i\times\Omega_i}$ denotes the identity matrix, and we define
\begin{equation*}
	S^{(\emptyset)}:=\begin{bmatrix} \frac{\sqrt{p_1p_2}}{2} -\varepsilon_1 & \eta-\frac{1}{2} \\ \eta-\frac{1}{2} & \frac{\sqrt{p_1p_2}}{2} -\varepsilon_2 \end{bmatrix}, \quad S^{(z)}:=\begin{bmatrix} \frac{\sqrt{p_1p_2}}{2} -\varepsilon_1 c_{1,1}^{(z)} & \eta c_{1,2}^{(z)} \\ \eta c_{2,1}^{(z)}& \frac{\sqrt{p_1p_2}}{2} -\varepsilon_2 c_{2,2}^{(z)} \end{bmatrix} \ (z\ne \emptyset),
\end{equation*}
where we are using the following notation:
\begin{equation*}
	c_{i,j}^{(z)}:=\prod_{\ell\in z}\big(-c_i^{(\ell)}c_j^{(\ell)}\big).
\end{equation*}
We have already verified that $S^{(\emptyset)}, S^{(\{1\})}\succcurlyeq 0$, so that we now consider $S^{(z)}$ with $z\ne\emptyset,\{1\}$.
By \eqref{psd0}, and since $\bigl|c_{1,1}^{(z)}\bigr|,\bigl|c_{2,2}^{(z)}\bigr|\leqslant 1$, the diagonal entries of $S^{(z)}$ are nonnegative.
Thus, $S^{(z)}\succcurlyeq 0$ if and only if
\begin{equation}\label{det:general}
	\left(\frac{\sqrt{p_1p_2}}2 -\varepsilon_1 c_{1,1}^{(z)} \right) \!\! \left( \frac{\sqrt{p_1p_2}}2-\varepsilon_2 c_{2,2}^{(z)} \right)  \geqslant \big(\eta c_{1,2}^{(z)}\big)^2.
\end{equation}
To prove the bound \eqref{bound}, we set $\varepsilon_2=0$.
In this case, using \eqref{epsilon-eta} with $\varepsilon_2=0$, together with $c_{1,1}^{(z)}=(-1)^{|z|}\bigl|c_{1,1}^{(z)}\bigr|$, $(c_{1,2}^{(z)})^2=\bigl|c_{1,1}^{(z)}\bigr|\bigl|c_{2,2}^{(z)}\bigr|$, we rewrite \eqref{det:general} as follows:
\begin{equation*}
	p_1p_2\geqslant \bigl|c_{1,1}^{(z)}\bigr| \left( q_2^2\,\bigl|c_{2,2}^{(z)}\bigr| +(-1)^{|z|} (p_1-p_2)p_2 \right).
\end{equation*}
We note that $p/(1-p)$ is increasing in $p\in(0,1/2]$, so that $p_i^{(\ell)}/q_i^{(\ell)}\leqslant p_i/q_i\leqslant 1$ for all $\ell\in[n]$ and $i=1,2$.
Thus, if $|z|$ is odd then
\begin{equation}\label{det:odd}
	\bigl|c_{1,1}^{(z)}\bigr| \left( q_2^2\,\bigl|c_{2,2}^{(z)}\bigr| -(p_1-p_2)p_2 \right) \leqslant \frac{p_1}{q_1} \left( q_2^2 \cdot \frac{p_2}{q_2} -(p_1-p_2)p_2 \right) = p_1p_2,
\end{equation}
and if $|z|$ is even then
\begin{equation}\label{det:even}
	\bigl|c_{1,1}^{(z)}\bigr| \left( q_2^2\,\bigl|c_{2,2}^{(z)}\bigr| +(p_1-p_2)p_2 \right) \leqslant \frac{p_1^2}{q_1^2} \left( q_2^2 \cdot\frac{p_2^2}{q_2^2} +(p_1-p_2)p_2 \right) = \frac{p_1^3p_2}{q_1^2} \leqslant p_1p_2.
\end{equation}
It follows that $S\succcurlyeq 0$, so that \eqref{values} (where $\varepsilon_2=0$ and $\varepsilon_1,\eta$ are given by \eqref{epsilon-eta}) provides a feasible solution to (D) with objective value $\sqrt{p_1p_2}$.
The proof of Lemma~\ref{lem:ineq} is complete.

\subsection{Reduction to kernels when $p_1\leqslant 1/2$}

For $i=1,2$, $E\subset\Omega_i$, and $z\in 2^{[n]}$, let $E|_z:=\{x\cap z\in\Omega_i:x\in E\}$.
Note that $\Omega_i|_z$ is a copy of $2^{z}$ in $\Omega_i$.
Recall $w=w_{\bm{p}_1,\bm{p}_2}$ from \eqref{w}.
This subsection is devoted to the proof of the following result:

\begin{lem}\label{kernel}
Suppose that $\mu_1(U_1)\mu_2(U_2)=p_1p_2$.
If $p_1\leqslant 1/2$, then
\begin{equation}\label{kernel-structure}
	U_i=\{x\sqcup y:x\in U_i|_w,\, y\in \Omega_i|_{[n]\backslash w}\}=U_i|_w\times \Omega_i|_{[n]\backslash w} \qquad (i=1,2).
\end{equation}
(Here, we identify $\Omega_i|_w \times \Omega_i|_{[n]\backslash w}$ with $\Omega_i$.)
In particular, $U_1|_w,U_2|_w$ are cross-intersecting.
Moreover, if $p_1=p_2=1/2$, then $\bigl|U_1|_w\bigr|=\bigl|U_2|_w\bigr|=2^{|w|-1}$.
\end{lem}

The cross-intersecting families $U_1|_w,U_2|_w$ are called the \emph{kernels} of $U_1,U_2$.
We shall prove the uniqueness of the optimal configuration for $U_1,U_2$ as a consequence of the corresponding result for the kernels; see Subsection~\ref{subsec:1/2>=p_1}.
For the rest of this subsection, we continue to assume that $p_1\leqslant 1/2$, and retain the notation of Subsection~\ref{sec: proof of bound when p_1 at most 1/2}.

For $i=1,2$, let $\bm{x}_i\in\mathbb{R}^{\Omega_i}$ be the characteristic vector of $U_i$, and let $\bm{v}_i^{(z)}\in\mathbb{R}^{\Omega_i}$ ($z\in 2^{[n]}$) be the column vectors of $V_i$.
The $\bm{v}_i^{(z)}$ form a basis of $\mathbb{R}^{\Omega_i}$, so that we can write
\begin{equation}\label{characteristic vector}
	\frac{1}{\sqrt{\mu_i(U_i)}}\bm{x}_i=\sum_{z\in 2^{[n]}} \theta_i^{(z)}\bm{v}_i^{(z)}=V_i\,\Theta_i,
\end{equation}
where $\Theta_i:=\bigl[\theta_i^{(z)}\bigr]_{z\in 2^{[n]}}$ denotes the column vector indexed by $2^{[n]}$ whose $z$-entry is $\theta_i^{(z)}$ for $z\in  2^{[n]}$.
Assume from now on that $\mu_1(U_1)\mu_2(U_2)=p_1p_2$, and let $X=X_{U_1,U_2}$ be as in Section~\ref{sec: general problem}.
Consider the symmetric matrix $S$ with $\varepsilon_2=0$, and recall that $S$ corresponds to a feasible solution to (D).
Then, we must have $S\bullet X=0$ (see the comment before Theorem~\ref{sdp-thm}).
Thus, it follows from \eqref{characteristic vector} that
\begin{align}
	0 &= S\bullet X \label{psd_optimal} \\
	&= S\bullet \!\left( \begin{bmatrix} V_1& 0 \\ 0 & V_2 \end{bmatrix} \!\! \begin{bmatrix} \Theta_1 \\ \Theta_2 \end{bmatrix} \!\! \begin{bmatrix} \Theta_1 \\ \Theta_2 \end{bmatrix}^{\!\mathsf{T}} \!\! \begin{bmatrix} V_1 & 0 \\ 0 & V_2 \end{bmatrix}^{\!\mathsf{T}} \right) \notag \\
	&= \left( \begin{bmatrix} V_1& 0 \\ 0 & V_2 \end{bmatrix}^{\!\mathsf{T}} \!\! S \begin{bmatrix} V_1 & 0 \\ 0 & V_2 \end{bmatrix}\right)\! \bullet \! \left(\begin{bmatrix} \Theta_1 \\ \Theta_2 \end{bmatrix} \!\! \begin{bmatrix} \Theta_1 \\ \Theta_2 \end{bmatrix}^{\!\mathsf{T}}\right) \notag \\
	&= \sum_{z\in 2^{[n]}} S^{(z)} \bullet \! \left( \begin{bmatrix} \theta_1^{(z)} \\ \theta_2^{(z)} \end{bmatrix} \!\! \begin{bmatrix} \theta_1^{(z)} \\ \theta_2^{(z)} \end{bmatrix}^{\!\mathsf{T}}\right). \notag
\end{align}
Observe that either $p_1^{(\ell)}/q_1^{(\ell)}<p_1/q_1$ or $p_2^{(\ell)}/q_2^{(\ell)}<p_2/q_2$ if $\ell\not\in w$.
Thus, if $z\not\in 2^w$ then \eqref{det:odd}, \eqref{det:even} are strict inequalities, i.e., $S^{(z)}\succ 0$ (positive definite).
It follows from \eqref{psd_optimal} that $\theta_1^{(z)}=\theta_2^{(z)}=0$ whenever $z\not\in 2^w$.
Next, let $z\in 2^w$.
Note that $c_i^{(\ell)}=c_i:=c_i^{(1)}$ for all $\ell\in w$.
Thus, by definition, the $x$-entry of $\bm{v}_i^{(z)}$ is given by $\prod_{\ell\in z\backslash x} c_i^{(\ell)}\prod_{k\in x\cap z}(-1/c_i^{(k)})=c_i^{|z\backslash x|}(-1/c_i)^{|x\cap z|}$ for $x\in \Omega_i$.
Thus, for $x,y\in\Omega_i$, the $x$-entry and the $y$-entry of $\bm{v}_i^{(z)}$ are identical whenever $x\cap z=y\cap z$. 
By these comments, it follows that, for $x,y\in\Omega_i$, the $x$-entry and the $y$-entry of $\bm{x}_i$ are identical whenever $x\cap w=y\cap w$.
Thus, for $x\in \Omega_i$, we have $x\in U_i$ if and only if $x\cap w\in U_i$.
This proves \eqref{kernel-structure}.
In particular, we have $U_i|_w=U_i\cap (\Omega_i|_w)$, so that $U_1|_w,U_2|_w$ are cross-intersecting.

Finally, suppose that $p_1=p_2=1/2$, and let $E:=\{w\backslash x:x\in U_1|_w\}$, viewed as a subset of $\Omega_2|_w$.
Then, $U_2|_w\cap E=\emptyset$ since $U_1|_w,U_2|_w$ are cross-intersecting, and therefore
\begin{equation*}
	\bigl|U_1|_w\bigr|\cdot \bigl|U_2|_w\bigr| \leqslant \bigl|U_1|_w\bigr| \cdot\bigl|(\Omega_2|_w)\backslash E\bigr|=\bigl|U_1|_w\bigr|\cdot\bigl(2^{|w|}-\bigl|U_1|_w\bigr|\bigr)\leqslant (2^{|w|-1})^2,
\end{equation*}
with equality if and only if $\bigl|U_1|_w\bigr|=\bigl|U_2|_w\bigr|=2^{|w|-1}$.
On the other hand, it follows from \eqref{kernel-structure} that $\mu_i(U_i)=(1/2)^{|w|}\bigl|U_i|_w\bigr|$ for $i=1,2$.
Thus, we have $\bigl|U_1|_w\bigr|\cdot \bigl|U_2|_w\bigr|=(2^{|w|-1})^2$ since $\mu_1(U_1)\mu_2(U_2)=(1/2)^2$, and this completes the proof of Lemma~\ref{kernel}.

\subsection{Proof of Theorem~\ref{weighted-thm} when $p_1\leqslant 1/2$}
\label{subsec:1/2>=p_1}

In this subsection, we prove Theorem~\ref{weighted-thm} under the assumption that $p_1\leqslant 1/2$.
To this end, we use Lemma~\ref{kernel} as follows.
For the rest of this subsection, assume that $\mu_1(U_1)\mu_2(U_2)=p_1p_2$, and recall \eqref{kernel-structure}.
For $i=1,2$, let $\bm{p}'_i:=(p_i^{(\ell)}:\ell\in w)$.
Observe that $\mu_i(U_i)=\mu_{\bm{p}_i'}(U_i|_w)$ for $i=1,2$, so that $U_1|_w,U_2|_w$ (which are cross-intersecting) are optimal with respect to $\mu_{\bm{p}_1'},\mu_{\bm{p}_2'}$.
Thus, we may further assume in the following that
\begin{equation*}
	w=[n],
\end{equation*}
i.e., $p_i^{(1)}=\dots=p_i^{(n)}=p_i$ for each of $i=1,2$.
In particular, we now have $c_{i,j}^{(z)}=(-c_ic_j)^{|z|}$ for $z\in 2^{[n]}$, where we recall $c_i=\sqrt{p_i/q_i}$ for $i=1,2$.

\subsubsection{The case when $p_1=p_2=1/2$}

Here, we also exclude the case where $n=|w|\geqslant 3$ (cf.~Examples~\ref{ex-3} and \ref{ex-4}).
By Lemma~\ref{kernel} (with $w=[n]$), we have $|U_1|=|U_2|=2^{n-1}$.
If $n\leqslant 2$ then this is possible only when $U_1=U_1^{(\ell)},U_2=U_2^{(\ell)}$ for some $\ell\in w$.

\subsubsection{The case when $1/2=p_1>p_2$}\label{case:1/2=p_1>p_2}

We carefully look at the proof of Lemma~\ref{kernel} in more detail (with $w=[n]$).
Recall the matrix $S$ from Subsection~\ref{sec: proof of bound when p_1 at most 1/2} with $\varepsilon_2=0$, and note that $\varepsilon_1>0$ in this case (cf.~\eqref{epsilon-eta}).
If $z\in\binom{[n]}{1}$, then equality holds in \eqref{det:odd} and thus in \eqref{det:general}, so that $S^{(z)}\not\succ 0$ (but $S^{(z)}\succcurlyeq 0$).
Similarly, if $z\in\binom{[n]}{2}$ then equality holds in \eqref{det:even} and thus in \eqref{det:general}.
On the other hand, if $|z|\geqslant 3$, then since $\bigl|c_{2,2}^{(z)}\bigr|=(p_2/q_2)^{|z|}$ and $p_2/q_2<1$, we have $\bigl|c_{2,2}^{(z)}\bigr|<p_2/q_2$ in \eqref{det:odd} if $|z|$ is odd, and $\bigl|c_{2,2}^{(z)}\bigr|<p_2^2/q_2^2$ in \eqref{det:even} if $|z|$ is even, so that \eqref{det:general} is a strict inequality, i.e., $S^{(z)}\succ 0$.
It follows that $S^{(z)}\succ 0$ precisely when $|z|\geqslant 3$.

Now, recall the matrix $X=X_{U_1,U_2}$.
We have $S\bullet X=0$, so that it follows from \eqref{psd_optimal} and the above comments that $\theta_1^{(z)}=\theta_2^{(z)}=0$ in \eqref{characteristic vector} whenever $|z|\geqslant 3$.
In particular, $\bm{x}_1$ is a linear combination of the $\bm{v}_1^{(z)}$ with $|z|\leqslant 2$.
For $z\in2^{[n]}$, let $\bm{y}^{(z)}\in\mathbb{R}^{\Omega_1}$ be the characteristic vector of the family $U_1^{(z)}:=\{x\in\Omega_1:z\subset x\}$.
Observe that the $\bm{y}^{(z)}$ are the column vectors of
\begin{equation*}
	\begin{bmatrix} 1 & 0 \\ 1 & 1 \end{bmatrix}\otimes\dots \otimes\begin{bmatrix} 1 & 0 \\ 1 & 1 \end{bmatrix} \quad (n\ \text{copies}).
\end{equation*}
Moreover, since the $\bm{v}_1^{(z)}$ are the column vectors of
\begin{equation*}
	\begin{bmatrix} 1 & c_1 \\ 1 & -\frac{1}{c_1} \end{bmatrix}\otimes\dots \otimes\begin{bmatrix} 1 & c_1 \\ 1 & -\frac{1}{c_1} \end{bmatrix} \quad (n\ \text{copies})
\end{equation*}
in this case, and since
\begin{equation*}
	\begin{bmatrix} c_1 \\ -\frac{1}{c_1} \end{bmatrix}=c_1 \begin{bmatrix} 1 \\ 1 \end{bmatrix} -\left( c_1+\frac{1}{c_1} \right) \!\begin{bmatrix} 0 \\ 1 \end{bmatrix},
\end{equation*}
it follows that $\bm{v}_1^{(z)}$ is a linear combination of the $\bm{y}^{(z')}$ with $z'\subset z$; more specifically,
\begin{equation*}
	\bm{v}_1^{(z)}=\sum_{z'\subset z}c_1^{|z\backslash z'|}\left(-c_1-\frac{1}{c_1}\right)^{\!\!|z'|}\bm{y}^{(z')} \qquad (z\in 2^{[n]}).
\end{equation*}
(In fact, $c_1=1$ here, but the discussions in this subsection will be used again in later (sub)sections, where we have $c_1<1$.)
Thus, $\bm{x}_1$ can now be written as
\begin{equation}\label{x_1 in terms of y^z}
	\bm{x}_1=\sum_{\substack{z\in 2^{[n]} \\ |z|\leqslant 2 }}\lambda^{(z)} \bm{y}^{(z)}.
\end{equation}

We claim that $U_1$ is intersecting.
Suppose the contrary, and pick $x,y\in U_1$ such that $x\cap y=\emptyset$.
Since $y\subset [n]\backslash x$, it follows from the optimality that $[n]\backslash x\in U_1$.
Moreover, observe that the $(x,[n]\backslash x)$-entry of $A_{1,1}$ is $1$ in this case.
Since $\varepsilon_1>0$, this implies that $Z\bullet X>0$ (cf.~\eqref{values}).
On the other hand, since equality holds in \eqref{bound}, we must have $Z\bullet X=0$ (see the comment before Theorem~\ref{sdp-thm}).
This is a contradiction.
Thus, $U_1$ is intersecting.

Let $\Lambda_j:=\bigl\{z\in\binom{[n]}{j}:\lambda^{(z)}\neq 0\bigr\}$ for $j=1,2$, and recall that the vectors appearing in \eqref{x_1 in terms of y^z} are $0$-$1$ vectors.
Since $U_1$ is intersecting, the $\emptyset$-entry of $\bm{x}_1$ is $0$, so that $\lambda^{(\emptyset)}=0$.
Similarly, by looking at the $\{\ell\}$-entry of $\bm{x}_1$ for $\ell\in[n]$, we find that $|\Lambda_1|\leqslant 1$.
We now have the following four cases:
(i) $|\Lambda_1|=0$, $|\Lambda_2|\geqslant 2$; (ii) $|\Lambda_1|=0$, $|\Lambda_2|=1$; (iii) $|\Lambda_1|=1$, $|\Lambda_2|\geqslant 2$; (iv) $|\Lambda_1|=1$, $|\Lambda_2|\leqslant 1$.
For (i), there are distinct $\ell,\ell',\ell''\in [n]$ such that $\{\ell,\ell'\},\{\ell,\ell''\}\in\Lambda_2$, and we have $\lambda^{(\{\ell,\ell'\})}=\lambda^{(\{\ell,\ell''\})}=1$.
We also have $\lambda^{(\{\ell',\ell''\})}\in\{0,1\}$, but then the $\{\ell,\ell',\ell''\}$-entry of $\bm{x}_1$ is $2$ or $3$, a contradiction.
For (iii), there are distinct $\ell,\ell',\ell''\in [n]$ such that $\Lambda_1=\bigl\{\!\{\ell\}\!\bigr\}$ and $\{\ell,\ell'\},\{\ell,\ell''\}\in\Lambda_2$.
We have $\lambda^{(\{\ell\})}=1$ and $\lambda^{(\{\ell,\ell'\})}=\lambda^{(\{\ell,\ell''\})}=-1$, but then the $\{\ell,\ell',\ell''\}$-entry of $\bm{x}_1$ is $-1$, a contradiction.
For (ii), there are $\ell,\ell'\in[n]$ such that $\Lambda_2=\bigl\{\!\{\ell,\ell'\}\!\bigr\}$, so that $\lambda^{(\{\ell,\ell'\})}=1$, i.e., $U_1=U_1^{(\{\ell,\ell'\})}$, and we have $U_2=\{x\in \Omega_2:x\cap\{\ell,\ell'\}\neq\emptyset\}$ by the optimality.
In this case, since $p_1=1/2$ and $q_2<1$, it follows that
\begin{equation*}
	\mu_1(U_1)\mu_2(U_2)=p_1^2(1-q_2^2)=p_1p_2\cdot p_1(1+q_2)<p_1p_2,
\end{equation*}
a contradiction.
Thus, we are left with (iv).
Let $\Lambda_1=\bigl\{\!\{\ell\}\!\bigr\}$, so that $\lambda^{(\{\ell\})}=1$.
If $|\Lambda_2|=1$, then there is an $\ell'\in[n]$ such that $\Lambda_2=\bigl\{\!\{\ell,\ell'\}\!\bigr\}$ and $\lambda^{(\{\ell,\ell'\})}=-1$, i.e., $U_1=U_1^{(\ell)}\backslash U_1^{(\{\ell,\ell'\})}$, but this is impossible, since $[n]\in U_1$ by the optimality.
Thus, we have $|\Lambda_2|=0$ and therefore $U_1=U_1^{(\ell)}$, from which it follows that $U_2=U_2^{(\ell)}$ as well.

\subsubsection{The case when $1/2>p_1$}

In this case, we consider the matrix $S$ with $\varepsilon_2>0$.
Recall again that we are assuming that $w=[n]$.
If $z\in\binom{[n]}{1}$, then equality holds in \eqref{det:general} for all $0\leqslant \varepsilon_2\leqslant \sqrt{p_1p_2}/2$ (see \eqref{psd2} and the comment after \eqref{epsilon-eta}).
If $|z|\geqslant 2$, then since $p_1/q_1,p_2/q_2<1$, we have $\bigl|c_{1,1}^{(z)}\bigr|<p_1/q_1$ and $\bigl|c_{2,2}^{(z)}\bigr|<p_2/q_2$ in \eqref{det:odd} if $|z|$ is odd, and $p_1^3p_2/q_1^2<p_1p_2$ in \eqref{det:even} if $|z|$ is even, from which it follows that we can choose a sufficiently small $\varepsilon_2>0$ so that \eqref{det:general} is a strict inequality whenever $|z|\geqslant 2$.
Thus, for such an $\varepsilon_2$, \eqref{values} (together with \eqref{epsilon-eta}) provides a feasible solution to (D) with objective value $\sqrt{p_1p_2}$, and we have $S^{(z)}\succ 0$ precisely when $|z|\geqslant 2$.

Note that the $(x,y)$-entry of $A_{i,i}$ ($x,y\in\Omega_i$, $i=1,2$) is positive whenever $x\cap y=\emptyset$.
Thus, since $\varepsilon_1,\varepsilon_2>0$, it follows from $Z\bullet X=0$ (cf.~\eqref{values}) that $U_1,U_2$ are intersecting.
We again have \eqref{x_1 in terms of y^z}, but the sum is over $z\in 2^{[n]}$ with $|z|\leqslant 1$, i.e., $\lambda^{(z)}=0$ for all $z\in\binom{[n]}{2}$.
Since $U_1$ is intersecting, it follows that there is an $\ell\in[n]$ such that $U_1=U_1^{(\ell)}$, and then we have $U_2=U_2^{(\ell)}$ as well.

The proof of Theorem~\ref{weighted-thm} is complete when $p_1\leqslant 1/2$.

\subsection{Proof of Theorem~\ref{weighted-thm} when $p_1>1/2$}
\label{subsec:bound}

In this subsection, we deal with the case when $p_1>1/2$.
The matrix $A_{1,1}$ now has negative entries, so that the SDP method from Section~\ref{sec: general problem} does not work in this case.
Instead, we invoke a combinatorial lemma due to Fishburn et al.~\cite{FFFLO1986SIAM}.
For convenience, we include a proof here.
Let $\Omega:=2^{[n]}$ be as in Theorem~\ref{FFFLO}.
We call $U\subset \Omega$ a \emph{co-complex} if $x\in U$ and $x\subset y\in\Omega$ imply $y\in U$.

\begin{lem}[\cite{FFFLO1986SIAM}]\label{monotone}
Let $\bm{p}=(p^{(1)},\dots,p^{(n)}),\tilde{\bm{p}}=(\tilde{p}^{(1)},\dots,\tilde{p}^{(n)})$ be probability vectors, and assume that $p^{(1)}>\tilde{p}^{(1)}$, and that $p^{(\ell)}=\tilde{p}^{(\ell)}$ for $\ell\geqslant 2$.
If $U\subset\Omega$ is a co-complex, then $\mu_{\bm{p}}(U) \leqslant (p^{(1)}/\tilde{p}^{(1)})\mu_{\tilde{\bm{p}}}(U)$.
Moreover, if equality holds, then $1\in x$ for all $x\in U$.
\end{lem}

\begin{proof}
We have $U=U'\sqcup U''\sqcup U'''$, where 
\begin{equation*}
	U':=\{x\in U:1\not\in x\},  \quad U'':=\{\{1\}\sqcup x: x\in U'\}, \quad U''':=U\backslash (U'\sqcup U'').
\end{equation*}
There is a bijection from $U'$ to $U''$ which sends $x\in U'$ to $y:=\{1\}\sqcup x\in U''$.
Then, letting $\mu_{\bm{p}}(\{x\})=(1-p^{(1)})\pi_x$, we have $\mu_{\bm{p}}(\{y\})=p^{(1)}\pi_x$, $\mu_{\tilde{\bm{p}}}(\{x\})=(1-\tilde{p}^{(1)})\pi_x$, and $\mu_{\tilde{\bm{p}}}(\{y\})=\tilde{p}^{(1)}\pi_x$, from which it follows that $\mu_{\bm{p}}(\{x,y\})=\mu_{\tilde{\bm{p}}}(\{x,y\})=\pi_x$.
Thus, we have $\mu_{\bm{p}}(U'\sqcup U'')=\mu_{\tilde{\bm{p}}}(U'\sqcup U'')$.
Next, observe that $1\in z$ for all $z\in U'''$.
Thus, letting $\mu_{\tilde{\bm{p}}}(U''')=\tilde{p}^{(1)}\sigma$, we have $\mu_{\bm{p}}(U''')=p^{(1)}\sigma=(p^{(1)}/\tilde{p}^{(1)})\mu_{\tilde{\bm{p}}}(U''')$.
It follows that
\begin{equation*}
	\mu_{\bm{p}}(U)=\mu_{\bm{p}}(U'\sqcup U'')+\mu_{\bm{p}}(U''')=\mu_{\tilde{\bm{p}}}(U'\sqcup U'')+ \frac{p^{(1)}}{\tilde{p}^{(1)}}\cdot\mu_{\tilde{\bm{p}}}(U''')\leqslant \frac{p^{(1)}}{\tilde{p}^{(1)}}\cdot\mu_{\tilde{\bm{p}}}(U).
\end{equation*}
If equality holds, then $\mu_{\tilde{\bm{p}}}(U'\sqcup U'')=0$, so that $U'=\emptyset$, as desired.
\end{proof}

We now return to the proof of Theorem~\ref{weighted-thm} when $p_1>1/2$.
Let $\tilde{p}_i:=\max\{p_i^{(\ell)}:\ell\geqslant 2\}$ for $i=1,2$.
Recall that we are assuming (besides \eqref{assumption on p}) that $p_1^{(\ell)},p_2^{(\ell)}\leqslant 1/2$ for $\ell\geqslant 2$, so that we have $\tilde{p}_1,\tilde{p}_2\leqslant 1/2$.
For $i=1,2$, let
\begin{equation*}
	\tilde{\bm{p}}_i=(\tilde{p}_i^{(1)},\dots,\tilde{p}_i^{(n)}):=(\tilde{p}_i,p_i^{(2)},p_i^{(3)},\dots,p_i^{(n)}).
\end{equation*}
Note that $\tilde{\bm{p}}_1,\tilde{\bm{p}}_2$ also satisfy \eqref{assumption on p}.
We also let $\tilde{w}:=w_{\tilde{\bm{p}}_1,\tilde{\bm{p}}_2}=\bigl\{\ell\in [n]:(\tilde{p}_1^{(\ell)},\tilde{p}_2^{(\ell)})=(\tilde{p}_1,\tilde{p}_2)\bigr\}$.
Since $\tilde{p}_1,\tilde{p}_2\leqslant 1/2$, it follows from Lemma~\ref{lem:ineq} that
\begin{equation}\label{eq1}
	\mu_{\tilde{\bm{p}}_1}(U_1)\mu_{\tilde{\bm{p}}_2}(U_2) \leqslant \tilde{p}_1\tilde{p}_2.
\end{equation}
Without loss of generality, we may assume that each of $U_1,U_2$ is a  co-complex.
Then, by Lemma~\ref{monotone}, we have
\begin{equation}\label{eq2}
	\mu_i(U_i) \leqslant \frac{p_i}{\tilde{p}_i} \cdot\mu_{\tilde{\bm{p}}_i}(U_i) \qquad (i=1,2).
\end{equation}
Combining \eqref{eq1} and \eqref{eq2}, we have
\begin{equation}\label{eq4}
	\mu_1(U_1)\mu_2(U_2) \leqslant \frac{p_1p_2}{\tilde{p}_1\tilde{p}_2} \cdot\mu_{\tilde{\bm{p}}_1}(U_1)\mu_{\tilde{\bm{p}}_2}(U_2) \leqslant p_1p_2,
\end{equation}
i.e., the bound \eqref{bound} holds.

For the rest of this subsection, assume that equality holds in \eqref{eq4}.
Then, equality also holds in \eqref{eq2} for $i=1,2$.
Since $p_1>1/2\geqslant\tilde{p}_1$, it follows from Lemma~\ref{monotone} that $U_1\subset U_1^{(1)}$.

\subsubsection{The case when $p_2>1/2$}

By Lemma~\ref{monotone}, and since $p_2>1/2\geqslant\tilde{p}_2$, we have $U_2\subset U_2^{(1)}$ as well.
Thus, $U_1=U_1^{(1)}, U_2=U_2^{(1)}$ by the optimality.

\subsubsection{The case when $\tilde{p}_1<1/2$ or $\tilde{p}_2<1/2$}

By the results of Subsection~\ref{subsec:1/2>=p_1} (applied to $\tilde{\bm{p}}_1,\tilde{\bm{p}}_2$), and since equality holds in \eqref{eq1}, it follows that $U_1=U_1^{(\ell)},U_2=U_2^{(\ell)}$ for some $\ell\in \tilde{w}$.
Since $U_1\subset U_1^{(1)}$, we conclude that $\ell=1$, as desired.

\subsubsection{The case when $1/2=\tilde{p}_1=p_2=\tilde{p}_2$}

We first prove the uniqueness of the optimal configuration for the kernels:

\begin{claim}
If $\bm{p}_1=(p,1/2,\dots,1/2)$, $\bm{p}_2=(1/2,\dots,1/2)$ where $p=p_1>1/2$, then $U_1=U_1^{(1)},U_2=U_2^{(1)}$.
\end{claim}

\begin{proof}
Recall again that $U_1\subset U_1^{(1)}$.
Then, $U_2^{(1)}\subset U_2$ by the optimality, and therefore we can write $U_2=U_2^{(1)}\sqcup E$, where $E\subset \Omega_2|_{[n]\backslash\{1\}}$.
Let $F:=\{[n]\backslash x:x\in E\}$, viewed as a subset of $U_1^{(1)}$.
Then, $U_1\cap F=\emptyset$ since $U_1,E$ are cross-intersecting, so that
\begin{equation*}
	\mu_1(U_1)\leqslant \mu_1(U_1^{(1)})-\mu_1(F)=p-\frac{p|E|}{2^{n-1}}.
\end{equation*}
On the other hand,
\begin{equation*}
	\mu_2(U_2)=\mu_2(U_2^{(1)})+\mu_2(E)=\frac{1}{2}+\frac{|E|}{2^n}.
\end{equation*}
Thus, we have
\begin{equation*}
	\mu_1(U_1) \mu_2(U_2)\leqslant p\left(1-\frac{|E|}{2^{n-1}}\right)\cdot\frac{1}{2}\left( 1+ \frac{|E|}{2^{n-1}}\right)=\frac{p}{2}\left(1-\frac{|E|^2}{2^{2n-2}}\right)\leqslant\frac{p}{2},
\end{equation*}
and then equality implies that $U_2=U_2^{(1)}$ (i.e., $E=\emptyset$) and also $U_1=U_1^{(1)}$, as desired.
\end{proof}

We now consider the general case.
Since equality holds in \eqref{eq1}, it follows from Lemma~\ref{kernel} (applied to $\tilde{\bm{p}}_1,\tilde{\bm{p}}_2$) that $U_i=U_i|_{\tilde{w}} \times \Omega_i|_{[n]\backslash \tilde{w}}$ for $i=1,2$.
Thus, $U_1|_{\tilde{w}},U_2|_{\tilde{w}}$ are cross-intersecting, and moreover we have $\mu_i(U_i)=\mu_{\bm{p}_i''}(U_i|_{\tilde{w}})$ for $i=1,2$, where $\bm{p}_i'':=(p_i^{(\ell)}:\ell\in \tilde{w})$.
By Claim above (applied to $\bm{p}_1'',\bm{p}_2''$), it follows that $U_1|_{\tilde{w}}=U_1^{(1)}|_{\tilde{w}},U_2|_{\tilde{w}}=U_2^{(1)}|_{\tilde{w}}$, and therefore we have $U_1=U_1^{(1)},U_2=U_2^{(1)}$.

This completes the proof of Theorem~\ref{weighted-thm}.

\section{Proof of Theorem~\ref{thm 1/3}}\label{sec: proof of thm3}

We proceed as in the proof of Lemma~\ref{lem:ineq}.
Recall the symmetric matrix $S$ defined by \eqref{S} and \eqref{epsilon-eta}, and set $\varepsilon_2=\sqrt{p_1p_2}/2$.
Note that $\varepsilon_1=\sqrt{p_1p_2}/2$ and $\eta=1/2$ in this case.
Using $(c_{1,2}^{(z)})^2=\bigl|c_{1,1}^{(z)}\bigr|\bigl|c_{2,2}^{(z)}\bigr|$, we can rewrite \eqref{det:general} as follows:
\begin{equation}\label{det:e1=e2}
	p_1p_2 \left(\frac{1-c_{1,1}^{(z)}}{\bigl|c_{1,1}^{(z)}\bigr|}\right) \!\! \left(\frac{1-c_{2,2}^{(z)}}{\bigl|c_{2,2}^{(z)}\bigr|}\right)\geqslant 1.
\end{equation}
We note that $(1-t)/|t|$ is decreasing (resp.~increasing) in $t$ for $t>0$ (resp.~$t<0$).
Thus, in order to show that $S\succcurlyeq 0$, it suffices to verify \eqref{det:e1=e2} for $z\in 2^{[n]}$ with $|z|=1,2$. 

First, suppose that $|z|=1$, and let $z=\{\ell\}$.
In this case, we have
\begin{equation*}
	\frac{1-c_{i,i}^{(z)}}{\bigl|c_{i,i}^{(z)}\bigr|}=\frac{q_i^{(\ell)}}{p_i^{(\ell)}}+1=\frac1{p_i^{(\ell)}}.
\end{equation*}
Thus, \eqref{det:e1=e2} is equivalent to $p_1p_2\geqslant p_1^{(\ell)}p_2^{(\ell)}$, and this is just \eqref{weak assumption}.
Next, suppose that $|z|=2$, and let $z=\{\ell,\ell'\}$.
Then, we have
\begin{equation*}
	\frac{1-c_{i,i}^{(z)}}{\bigl|c_{i,i}^{(z)}\bigr|}=\frac{1}{\bigl|c_{i,i}^{(z)}\bigr|}-1=\frac{q_i^{(\ell)}q_i^{(\ell')}-p_i^{(\ell)}p_i^{(\ell')}}{p_i^{(\ell)}p_i^{(\ell')}}=\frac{1-p_i^{(\ell)}-p_i^{(\ell')}}{p_i^{(\ell)}p_i^{(\ell')}},
\end{equation*}
so that \eqref{det:e1=e2} is equivalent to
\begin{equation*}
	p_1p_2\bigl(1-p_1^{(\ell)}-p_1^{(\ell')}\bigr) \! \bigl(1-p_2^{(\ell)}-p_2^{(\ell')}\bigr) \geqslant p_1^{(\ell)}p_1^{(\ell')}p_2^{(\ell)}p_2^{(\ell')}.
\end{equation*}
This is certainly true since $p_1p_2\geqslant p_1^{(\ell)}p_2^{(\ell)}$ and $1-p_i^{(\ell)}-p_i^{(\ell')}\geqslant 1/3\geqslant p_i^{(\ell')}$.
Thus, it follows that $S\succcurlyeq 0$.
Moreover, it is clear that $Z\geqslant 0$ in this case, and therefore \eqref{values} provides a feasible solution to (D) with objective value $\sqrt{p_1p_2}$.
This proves the bound \eqref{bound}.

Now, assume that $\mu_1(U_1)\mu_2(U_2)=p_1p_2$, and recall the matrix $X=X_{U_1,U_2}$.
We note that the $(x,y)$-entry of $A_{i,i}$ ($x,y\in\Omega_i$) is positive whenever $x\cap y=\emptyset$.
Thus, it follows from $Z\bullet X=0$ that $U_1,U_2$ are intersecting.
Moreover, observe that \eqref{det:e1=e2} and thus \eqref{det:general} are strict inequalities unless $z\in\binom{w}{1}$ or ($p_1=p_2=1/3$ and) $z\in\binom{w}{2}$, where $w$ is as in Conjecture~\ref{relaxed version}.
We again have \eqref{x_1 in terms of y^z}, but the sum is over $z\in 2^w$ with $|z|\leqslant 2$.
We can now argue exactly as in Subsection~\ref{case:1/2=p_1>p_2} to conclude that $U_1=U_1^{(\ell)}$, $U_2=U_2^{(\ell)}$ for some $\ell\in w$.

\section*{Acknowledgments}

The authors thank Peter Frankl for telling them the reference \cite{FFFLO1986SIAM}.
They also thank the anonymous referees for comments and suggestions.
Hajime Tanaka was supported by JSPS KAKENHI Grant No.~25400034.
Norihide Tokushige was supported by JSPS KAKENHI Grant No.~25287031.

\end{document}